\documentclass[11pt]{article} 
\usepackage[left=1in,top=1in,right=1in,bottom=1in,nohead]{geometry}

\usepackage{amssymb,amsfonts,amsmath,amsthm}
\usepackage[numbers,square]{natbib}
\usepackage{color}
\usepackage{hyperref}
\hypersetup{backref,colorlinks=true,citecolor=blue,linkcolor=blue,urlcolor=blue}

\usepackage{aliascnt}
\newtheorem{theorem}{Theorem}[section]

\newaliascnt{result}{theorem} 

\aliascntresetthe{result} 

\newaliascnt{lemma}{theorem} 
\newtheorem{lemma}[lemma]{Lemma}
\aliascntresetthe{lemma} 

\newaliascnt{prop}{theorem}
\newtheorem{proposition}[prop]{Proposition}
\aliascntresetthe{prop}

\newaliascnt{cor}{theorem} 

\aliascntresetthe{cor} 

\newaliascnt{conj}{theorem}

\aliascntresetthe{conj} 

\newaliascnt{def}{theorem}
\newtheorem{definition}[def]{Definition}
\aliascntresetthe{def}

\renewcommand{\hat}[1]{\widehat{#1}}
\newcommand{\email}[1]{\href{mailto:#1}{#1}}
\newcommand{\R}{\mathbb{R}}
\newcommand{\E}{\mathbb{E}}
\newcommand{\V}{\mathbb{V}}
\renewcommand{\P}{\mathbb{P}}
\newcommand{\X}{\mathbb{X}}
\newcommand{\ep}{\varepsilon}

\newcommand{\thetaLOOnoj}{\hat\theta^{(i)}(\lambda)}

\newcommand{\op}{o_{\mathbb{P}}}

\newcommand{\lamCV}{\widehat{\lambda}_n}
\DeclareMathOperator*{\argmin}{argmin} 

\DeclareMathOperator*{\tr}{\textrm{tr}}

\newcommand{\norm}[1]{\left|\left| #1 \right|\right|}
\newcommand{\ols}{\hat\theta(0)}

\newcommand{\lasso}{\hat\theta(\lambda)}
\newcommand{\lassoj}{\hat\theta_j(\lambda)}

\newcommand{\Cin}{C_{i,n}}
\newcommand{\cmin}{c_{\min}}

\begin{document}

\title{Leave-one-out cross-validation is risk consistent for lasso}
\author{Darren Homrighausen\\Department of Statistics\\Colorado
  State University\\Fort Collins,
  Colorado\\\email{darrenho@stat.colostate.edu} \and Daniel
  J. McDonald \\ Department of Statistics\\Indiana
  University\\Bloomington, IN 47408\\\email{dajmcdon@indiana.edu}}
\date{\today}

\maketitle

\begin{abstract}
  The lasso procedure is ubiquitous in the statistical and signal
  processing literature, and as such, is the target
  of substantial theoretical and applied
  research. While much of this research focuses on the desirable
  properties that lasso possesses---predictive risk consistency,
  sign consistency, correct model selection---all of it has assumes
  that the tuning parameter is chosen in an oracle fashion. Yet, this
  is impossible in practice. Instead, data analysts must use the
  data twice, once to choose the tuning parameter and again to
  estimate the model. But only heuristics have ever justified such
  a procedure. To this end, we give the first definitive answer
  about the risk consistency of lasso when the smoothing parameter is
  chosen via cross-validation. We show that under some restrictions on the
  design matrix, the lasso estimator is still risk consistent with an
  empirically chosen tuning parameter. \\
  \vspace{1em}

\end{abstract}

\section{Introduction}

Since its introduction in the statistical~\citep{Tibshirani1996} and
signal processing~\citep{ChenDonoho1998} communities, the lasso has
become a fixture as both a data analysis tool \citep[for
example]{LeeZhu2010,ShiWahba2008} and as an object for deep
theoretical investigations
\citep{FuKnight2000,GreenshteinRitov2004,MeinshausenBuhlmann2006}.  To
fix ideas, suppose that the observational model is of the form
\begin{equation}
  Y = \X \theta + \sigma W.
  \label{eq:fullLinearModel}
\end{equation}
where $Y=(Y_1,\ldots,Y_n)^{\top}$ is the vector of responses 
and $\X \in \mathbb{R}^{n\times p}$ is the feature matrix, with rows
$(X_i^{\top})_{i=1}^n$, $W$ is a noise vector, and $\sigma$ is the
signal-to-noise ratio.
Under (\ref{eq:fullLinearModel}), the lasso estimator, 
$\widehat \theta(\lambda)$, is defined to be
the minimizer of the following functional:
\begin{equation}
  \widehat \theta(\lambda) := \argmin_{\theta} \frac{1}{2n} || Y-\X\theta ||_2^2 + \lambda || \theta ||_1.
  \label{eq:regularized}
\end{equation}
Here, $\lambda \geq 0$ is a tuning parameter controlling the trade-off
between fidelity to the data (small $\lambda$) and sparsity (large $\lambda$).
We tacitly assume that $\X$ has full column rank,
and thus, $\widehat \theta(\lambda)$
is the unique minimum.

Under conditions on the matrix $\X$, noise
vector $W$, and the parameter $\theta$, the optimal choice of
$\lambda$ leads to risk consistency \citep{GreenshteinRitov2004}.
However, arguably the most crucial aspect of any procedure's
performance is the selection of the tuning parameters.  Typically,
theory advocating the lasso's empirical properties specifies only the
rates. That is, this theory claims ``if $\lambda=\lambda_n$ goes to
zero at the correct rate, then $\widehat{\theta}(\lambda_n)$ will be
consistent in some sense.'' For the regularized problem
in~(\ref{eq:regularized}), taking $\lambda_n=o((\log(n)/n)^{1/2})$
gives risk consistency under very general conditions. However, this
type of theoretical guidance says nothing about the properties of the
lasso when the tuning parameter is chosen using the data.

There are several proposed techniques for choosing $\lambda$, such as
minimizing the empirical risk plus a penalty term based on the
 degrees of freedom \citep{ZouHastie2007,TibshiraniTaylor2012}
or using an adapted Bayesian information criterion \citep{WangLeng2007}.  
In many papers, \citep[for
example]{Tibshirani1996,GreenshteinRitov2004,HastieTibshirani2009,EfronHastie2004,ZouHastie2007,Tibshirani2011,GeerLederer2011}, the recommended technique for selecting
$\lambda$ is to choose $\lambda=\lamCV$ such that
$\lamCV$ minimizes a cross-validation estimator of the risk.

Some results supporting the use of cross-validation for statistical
algorithms other than lasso are known. For instance, kernel regression
\citep[Theorem 8.1]{Gyorfi2002}, $k$-nearest neighbors \citep[Theorem
8.2]{Gyorfi2002}, and various classification algorithms \citep{Schaffer1993} all
behave well with tuning parameters selected using the data.  Additionally, 
suppose we form the adaptive ridge regression estimator \citep{Grandvalet1998}
\begin{equation}
\argmin_{\theta,(\lambda_j)} \norm{Y - \X \theta}_2^2 + \sum_{j=1}^p \lambda_j \theta_j^2
\label{eq:adaptiveRidge}
\end{equation}
subject to the contraint $\lambda\sum_{j=1}^p  1/\lambda_j = p$.  Then the solution
to equation \eqref{eq:adaptiveRidge} is equivalent, under a reparameterization of $\lambda$,
to the solution to equation \eqref{eq:regularized}.  As ridge regression has been shown to have
good asymptotic properties under (generalized) cross-validation, there is reason to believe these properties
may carry over to lasso and cross-validation using this equivalence.
However, rigorous results for the lasso have yet to be developed.

The supporting theory for other methods indicates that there should
be corresponding theory for the lasso.  However, other results are not
so encouraging.  In particular, \citep{Shao1993} shows that
cross-validation is inconsistent for model selection.  As lasso
implicitly does model selection, and shares many connections with
forward stagewise regression \citep{EfronHastie2004}, this raises a
concerning possibility that lasso might similiarly be inconsistent
under cross-validation. Likewise, \citep{LengLin2006} shows that using
prediction accuracy (which is what cross-validation estimates) as a
criterion for choosing the tuning parameter fails to recover the
sparsity pattern consistently in an orthogonal design
setting. Furthermore, \citep{XuMannor2008} show that sparsity inducing
algorithms like lasso are not (uniformly) algorithmically stable. In
other words, leave-one-out versions of the lasso estimator are not
uniformly close to each other.  As shown in
\citep{BousquetElisseeff2002}, algorithmic stability is a sufficient,
but not necessary, condition for risk consistency.

These results taken as a whole leave the lasso in an unsatisfactory
position, with some theoretical results and generally accepted
practices advocating the use of cross-validation while others suggest
that it may not work.  Our result partially resolves this antagonism by
showing that, in some cases, the lasso with cross-validated tuning
parameter is indeed risk consistent.

In this paper we provide a first result about the risk consistency of
lasso with the tuning parameter selected by cross-validation under
some assumptions about $\X$.  In
Section~\ref{sec:notation-assumptions} we introduce our notation and
state our main theorem. In Section~\ref{sec:prelim-material} we state
some results necessary for our proof methods and in
Section~\ref{sec:proofs} we provide the proof. Lastly, in
Section~\ref{sec:discussion} we mention some implications of our main
theorem and some directions for future research.

\section{Notation, assumptions, and main results}
\label{sec:notation-assumptions}

The main assumptions we make for this paper ensure that the sequence $(X_i)_{i=1}^n$ is sufficiently regular.
These are

\textbf{Assumption A:}
\begin{equation}
C_n := \frac{1}{n} \sum_{i=1}^n X_i X_i^{\top} \rightarrow C,
\end{equation}
where $C$ is a positive definite matrix with  $\textrm{eigen}_{\min}(C) = \cmin > 0$,
and

\textbf{Assumption B:}
There exists a constant $C_X < \infty$ independent of $n$ such that
\begin{equation}
\norm{X_i}_2 \leq C_X.
\end{equation}

Note that Assumption A appears repeatedly in the literature  in 
various contexts \citep[for
example]{Tibshirani1996,FuKnight2000,OsbornePresnell2000,LengLin2006}.
Additionally, Assumption B is effectively equivalent to assuming $\max_i \{ ||X_i||_2, 1\leq i \leq n\} = O(1)$ as 
$n \rightarrow \infty$, which is also standard  \citep[for example]{ChatterjeeLahiri2011}.

We define the predictive risk and the leave-one-out cross-validation
estimator of risk to be
\begin{equation}
  R_n(\lambda) := \frac{1}{n}\mathbb{E} ||\X(\widehat\theta(\lambda) -
  \theta)||^2 + \sigma^2 
  =
  \mathbb{E} ||\widehat\theta(\lambda) -
  \theta)||_{C_n}^2 + \sigma^2 
  \label{eq:trueRisk}
\end{equation}
and  
\begin{equation}  
  \widehat R_n(\lambda) = \frac{1}{n} \sum_{i=1}^n (Y_i -
  X_i^{\top}\thetaLOOnoj)^2, 
    \label{eq:cvRisk}
\end{equation}
respectively.  Here we are using $\thetaLOOnoj$ to
indicate the lasso estimator $\widehat\theta(\lambda)$ 
computed using all but the $i^{th}$ observation.  Also,
we write the $\ell^2$-norm weighted by a matrix $A$ to be
$\norm{x}_A^2 = x^{\top}Ax$.

Lastly, let $\Lambda$ be a large, compact subset of $[0,\infty)$ the
specifics of which are unimportant. In practical situations, any
$\lambda \in [\max_j \widehat\theta_j(0),\infty)$ will
result in the same solution, namely $\widehat\theta_j(\lambda)=0$ for
all $j$, so any large finite upper bound is sufficient. Then define
\begin{align*}
  \lamCV := \argmin_{\lambda \in \Lambda} \widehat R_n(\lambda), &&
  \textrm{and}  
  && \lambda_n := \argmin_{\lambda \in \R^+} R_n(\lambda).
\end{align*}
For
$\widehat\theta(\lambda)$ to be consistent, it must hold that $\lambda
\rightarrow 0$ as $n\rightarrow\infty$. Hence, for some $N$, $n \geq
N$ implies $\lambda_n \in \Lambda \subset \R^+$.  Therefore, without loss of
generality, we assume that $\lambda_n \in \Lambda$ for all $n$.

We spend the balance of this paper discussing and proving the
following result:
\begin{theorem}[Main Theorem]
  \label{thm:mainTheorem} Suppose that Assumptions A and B hold and that
  there exists a 
  $C_{\theta}<\infty$ such that $\norm{\theta}_1 \leq C_{\theta}$.
  Also, suppose that $W_i \sim P_i$ are independently distributed and
  that there exists a $\tau<\infty$ independent of $i$ such that
  \[ 
  E_{P_i}\left[e^{tW_i}\right] \leq e^{\tau^2 t^2/2}
  \] 
  for all $t\in\mathbb{R}$.  Then
  \begin{equation} 
    R_n(\lamCV) - R_n(\lambda_n) \rightarrow 0.
    \label{eq:main}
  \end{equation}
\end{theorem}
Essentially, this result states that under some conditions on the
design matrix $\X$ and the noise vector $W$, the predictive risk of
the lasso estimator with tuning parameter chosen via cross-validation
converges to the predictive risk of the lasso estimator with the
oracle tuning parameter. In other words, the typical procedure for a
data analyst is asymptotically equivalent to the optimal procedure. We
will take $\P=\prod_i P_i$ to be the $n$-fold product distribution of
the $W_i$'s and use $\E$ to denote the expected value
with respect to this product measure.

To prove this theorem, we show that $\sup_{\lambda \in
\Lambda}|\widehat R_n(\lambda) - R_n(\lambda)| \rightarrow 0$ in
probability.  Then (\ref{eq:main}) follows as
\begin{align*}
  R_n(\lamCV) - R_n(\lambda_n) 
  & =
  \left(R_n(\lamCV) - \widehat R_n(\lamCV)\right) +  
  \left(\widehat R_n(\lamCV) - R_n(\lambda_n)\right)  \\
  & \leq
  \left(R_n(\lamCV) - \widehat R_n(\lamCV)\right) + 
  \left(\widehat R_n(\lambda_n) - R_n(\lambda_n)\right) \\
  & \leq
  2 \sup_{\lambda\in\Lambda} \left(R_n(\lambda) - \widehat R_n(\lambda)\right) \\
  & = o_\P(1). 
\end{align*}
In fact, the term $R_n(\lamCV) - R_n(\lambda_n)$ 
is non-stochastic (the expectation in the risk
integrates out the randomness in the data) and therefore convergence in
probability implies sequential convergence and hence $o_\P(1)=o(1)$. 

We can write
\begin{align}
  \lefteqn{|R_n(\lambda) - \hat R_n(\lambda)| }\notag\\
  & =
  \Bigg|\frac{1}{n}\E ||\X\hat\theta(\lambda)||_2^2 + \frac{1}{n}||\X\theta||_2^2 -
  \frac{1}{n}2\E(\X\hat\theta(\lambda))^{\top}\X\theta + \sigma^2  \notag \\ 
  & \qquad - \frac{1}{n} \sum_{i=1}^n \left(Y_i^2  +
      (X_i^{\top}\thetaLOOnoj)^2 - 2Y_iX_i^{\top}\thetaLOOnoj
    \right)\Bigg| \notag \\ 
  & \leq
  \underbrace{\left| \frac{1}{n}\E ||\X\hat\theta(\lambda)||_2^2 - \frac{1}{n} \sum_{i=1}^n
      (X_i^{\top}\thetaLOOnoj)^2  \right|}_{(a)} +
  \underbrace{2\left| \frac{1}{n}\E(\X\hat\theta(\lambda))^{\top}\X\theta - \frac{1}{n}
      \sum_{i=1}^nY_iX_i^{\top}\thetaLOOnoj \right|}_{(b)} \notag\\
  \label{eq:decomp}
  &\qquad  + \underbrace{\left| \frac{1}{n}||\X\theta||_2^2  + \sigma^2 -  \frac{1}{n} \sum_{i=1}^n
      Y_i^2\right|}_{(c)}.
\end{align}
Our proof follows by addressing $(a)$, $(b)$, and $(c)$ in lexicographic order in Section
\ref{sec:proofs}. To show that each term converges in probability to
zero uniformly in $\lambda$, we will need a few preliminary results.

\section{Preliminary material}
\label{sec:prelim-material}

In this section, we present some definitions and lemmas which are
useful for proving risk consistency of the lasso with cross-validated
tuning parameter. First, we give some results regarding the
uniform convergence of measurable functions. Next, we use these
results to show that the leave-one-out lasso estimator converges
uniformly to the full-sample lasso estimator. Finally, we present a
concentration inequality for quadratic forms of sub-Gaussian random
variables. 

\subsection{Equicontinuity}
\label{sec:equicontinuity}

Our proof of Theorem \ref{thm:mainTheorem} uses a number of results
 relating uniform convergence with convergence in probability. The
essential message is that particular measurable functions behave
nicely over compact sets. Mathematically, such collections of functions are called
{\em stochastically equicontinuous}.

To fix ideas, we first present the
definition of stochastic equicontinuity in the context of statistical
estimation. Suppose that we are 
interested in estimating some functional of a parameter
$\beta$, $\overline{Q}_n(\beta)$,
using $\hat{Q}_n(\beta)$ where $\beta \in \mathcal{B}$.

\begin{definition}
  [Stochastic equicontinuity]
  If for every $\ep, \eta>0$ there exists a random variable
  $\Delta_n(\ep,\eta)$ and constant $n_0(\ep,\eta)$ such that for
  $n\geq n_0(\ep,\eta)$, $\P(|\Delta_n(\ep,\eta)|>\ep) < \eta$ and for
  each $\beta \in \mathcal{B}$ there is an open set $\mathcal{N}(\beta,\ep,\eta)$
  containing $\beta$ such that for $n\geq n_0(\ep,\eta)$,
  \[
  \sup_{\beta' \in \mathcal{N}(\beta,\ep,\eta)} \left|
    \hat{Q}_n(\beta')-\hat{Q}_n(\beta) \right| \leq
  \Delta_n(\ep,\eta),
  \]
  then we call $\{\hat{Q}_n\}$ \emph{stochastically equicontinuous over }$\mathcal{B}$.
\end{definition}

An alternative formulation of stochastic equicontinuity which is often
more useful can be found via a Lipschitz-type condition.
\begin{theorem}[Theorem 21.10 in~\citep{Davidson1994}]
  \label{thm:davidsonSE}
  Suppose there exists a random variable $B_n$ and a function $h$ such
  that $B_n=O_\P(1)$ and for all $\beta',\beta \in \mathcal{B}$,
  $|\hat{Q}_n(\beta') - \hat{Q}_n(\beta)| \leq B_n h(d(\beta',\beta))$,
  where $h(x) \downarrow 0$ as $x \downarrow 0$ and $d$ is a metric on
  $\mathcal{B}$. Then $\{\hat{Q}_n\}$ is stochastically equicontinuous.
\end{theorem}
The importance of stochastic equicontinuity is in showing uniform
convergence, as is expressed in the following two results. 
\begin{theorem}
  [Theorem 2.1 in~\citep{Newey1991}]
  \label{thm:neweyUniform}
  If $\mathcal{B}$ is compact, $|\hat{Q}_n(\beta) -
  \overline{Q}_n(\beta)| = o_\P(1)$ for each $\beta \in \mathcal{B}$, 
  $\{\hat{Q}_n\}$ is stochastically equicontinuous over $\mathcal{B}$, and
  $\{\overline{Q}_n\}$ is equicontinuous, then
  $\sup_{\beta \in \mathcal{B}}
  |\hat{Q}_n(\beta) - \overline{Q}_n(\beta)| = o_\P(1)$.
\end{theorem}
This theorem allows us to show uniform convergence of estimators
$\hat{Q}_n(\beta)$ of statistical functionals to $\overline{Q}_n(\beta)$
over compact sets $\mathcal{B}$. However, we may also be interested in
the uniform convergence of random quantities to each other. While one
could use the above theorem to show such a result, the following
theorem of \citep{Davidson1994} is often simpler.

\begin{theorem}
  [\citep{Davidson1994}]
  If $\mathcal{B}$ is compact, then    
  $ \sup_{\beta\in\mathcal{B}} G_n(\beta) = o_\P(1)$ if and only if
  $G_n(\beta)=o_\P (1)$ for each
  $\beta$ in a dense subset of $\mathcal{B}$ and $\{G_n(\beta)\}$ is
  stochastically equicontinuous. 
  \label{thm:davidsonUniform}
\end{theorem}

\subsection{Uniform convergence of lasso estimators}
\label{sec:unif-conv-lasso}

Using stochastic equicontinuity, we prove two lemmas about
lasso estimators which, while intuitive, are nonetheless
novel. The first shows that the lasso estimator converges uniformly over
$\Lambda$ to its expectation. The second shows that the lasso estimator
computed using the full sample converges in probability uniformly 
over $\Lambda$ to the lasso estimator computed with all but one observation.  

Before stating our lemmas, we include without proof some standard results about
uniform convergence of functions.
A function $f: [a,b] \rightarrow
\mathbb{R}$ has the Luzin $N$ property if, for all $N \subset [a,b]$
that has Lebesgue measure zero, $f(N)$ has Lebesgue measure zero as
well.  Also, a function $f$ is of bounded variation if and only if it
can be written as $f = f_1 - f_2$ for non-decreasing functions $f_1$
and $f_2$.

\begin{theorem}
  A function $f$ is absolutely continuous if and only if it is of
  bounded variation, continuous, and has the Luzin $N$ 
  property.
  \label{thm:absCont}
\end{theorem}

\begin{theorem}
  If a function $f:[a,b] \rightarrow \mathbb{R}$ is absolutely
  continuous, and hence differentiable almost everywhere, and satisfies
  $|f'(x)| \leq C_L$ for almost all $x \in [a,b]$ with respect to
  Lebesgue measure, then it is Lipschitz continuous with constant $C_L$.
  \label{thm:lipschitz}
\end{theorem}
Throughout this paper, we use $C_L$ as generic notation for a Lipschitz constant; its actual value changes from line to line.  The following result is useful for showing the uniform convergence $\lasso$.

\begin{proposition}
The random function $\lasso$ is Lipschitz continuous over $\Lambda$.  That is, there exists $C_L < \infty$
such that 
for any $\lambda,\lambda' \in \Lambda$,
\begin{equation}
\norm{\hat\theta(\lambda) - \hat\theta(\lambda') }_2 \leq C_L |\lambda - \lambda'|.
\end{equation}
Additionally, $C_L = O(1)$ as $n \rightarrow \infty$.
\label{prop:lassoLipschitz}
\end{proposition}

\begin{proof}
 The solution path of the lasso is piecewise linear over $\lambda$ with a finite number of
  `kinks.'  Using the notation developed in \citep[Section 3.1]{Tibshirani2013}, over each such interval, the
  nonzero entries in $\lasso$ behave as a linear function with slope $n(\X_{\mathcal{E}}^{\top}
  \X_{\mathcal{E}})^{-1}s_\mathcal{E}$, where $\mathcal{E} \subset
  \{1,\ldots,p\}$ is the set of the indices of the active variables, $s_\mathcal{E}$ is the vector of 
  signs, and
  $\X_{\mathcal{E}}$ is the feature matrix with columns restricted to the
  indices in $\mathcal{E}$.

  Therefore, as $ \norm{n(\X_{\mathcal{E}}^{\top}
    \X_{\mathcal{E}})^{-1}s_\mathcal{E}}_2 \leq \norm{n(\X_{\mathcal{E}}^{\top}
    \X_{\mathcal{E}})^{-1}}_2$,
  $\hat\theta(\lambda)$ is Lipschitz continuous with 
  \[
  C_L = \max_{\mathcal{E} \subset \{1,\ldots,p\}} 
  \norm{n(\X_{\mathcal{E}}^{\top}\X_{\mathcal{E}})^{-1}}_2
    \]
  By Assumption A, for any $\mathcal{E}$, $\frac{1}{n} \X_{\mathcal{E}}^{\top} \X_{\mathcal{E}}
  \rightarrow C_{\mathcal{E}}$.  Also, $\textrm{eigen}_{\min}(C_{\mathcal{E}}) \geq \cmin$
  for any $\mathcal{E}$.  Fix $\epsilon = \cmin/2$. Then, there exists
  an $N$ such that for all $n \geq N$ and any $\mathcal{E}$,
  \begin{equation}
  \frac{1}{n} \textrm{eigen}_{\min} (\X_{\mathcal{E}}^{\top}\X_{\mathcal{E}}) \geq \epsilon.
  \label{eq:lassoLipBound}
  \end{equation}
  Therefore, for $n$ large enough, $C_L \leq \frac{1}{\epsilon} < \infty$, which is independent of $n$.
\end{proof}
\begin{lemma}
For any $i = 1,\ldots,n$,
  \[
  \sup_{\lambda \in \Lambda} ||\hat\theta(\lambda) - \thetaLOOnoj||_2
  \xrightarrow{\P} 0. 
  \]
  \label{lem:lynchpin}
\end{lemma}

\begin{proof}
  The pointwise convergence of $ ||\hat\theta(\lambda) -
  \thetaLOOnoj||_2$ to zero follows by \citep[Theorem 1]{FuKnight2000}.  Hence,
  we invoke the consequent of Theorem~\ref{thm:davidsonUniform} as long as
  $||\hat\theta(\lambda) - \thetaLOOnoj||_2$ is stochastically
  equicontinuous.  For this, it is sufficient to show that
  $\hat\theta(\lambda)$ and $ \thetaLOOnoj$ are Lipschitz in the sense
  of Theorem \ref{thm:davidsonSE}.  This follows for both estimators by Proposition \ref{prop:lassoLipschitz}.
\end{proof}

\begin{lemma}
  \label{lem:a1a}
  For all $1\leq j\leq p$, $\{\lassoj\}$ is stochastically equicontinuous, $\{\E[\lassoj]\}$ is equicontinuous, and 
  $| \lassoj - \mathbb{E}\lassoj| = \op(1)$.
  Thus, 
  \[
  \sup_{\lambda\in\Lambda} | \lassoj - \mathbb{E}\lassoj| = \op(1).
  \]
  Furthermore, 
  \[
  \sup_{\lambda\in\Lambda} \norm{ \lasso - \mathbb{E}\lasso}_{C_n}^2 = \op(1),
  \]
  where this notation is introduced in equation \eqref{eq:trueRisk}.
\end{lemma}

\begin{proof}
  To show this claim, we use Theorem \ref{thm:neweyUniform}.
  For pointwise convergence, note that $\lasso$ converges in probability to an non-stochastic limit
  \citep[Theorem 1]{FuKnight2000}, call it $\theta(\lambda)$.  
  Also, $|\lassoj| \leq \norm{\ols}_1$, which is integrable.
By the Skorohod representation theorem, there exists random variables $\lassoj'$
such that $\lassoj' \rightarrow \theta(\lambda)$ almost surely and $\lassoj'$ has the same
distribution as  $\lassoj$ for each $n$.
By the dominated convergence theorem, 
\[
\lim \E \lassoj = \lim \E \lassoj' = \E \theta(\lambda) = \theta(\lambda).
\]
Therefore, $|\lassoj - \E \lassoj| \rightarrow 0$ in probability.

  Stochastic equicontinuity follows by Proposition
  \ref{prop:lassoLipschitz} and Theorem \ref{thm:davidsonSE}. 
  Hence, Theorem~\ref{thm:neweyUniform} is satisfied as long as
  $\{\E\lassoj\}$ is equicontinuous.  Observe that
 the expectation and differentiation operations commute for
  $\lasso$.  Therefore, the result follows by Proposition \ref{prop:lassoLipschitz}.

  Finally, we have
  \begin{align*}
    \norm{\lasso-\E\lasso}_{C_n}^2 
    & = 
    (\lasso - \E\lasso)^{\top} C_n (\lasso - \E\lasso) \\
    & \leq 
    \norm{\lasso - \E\lasso}_2 \norm{C_n (\lasso - \E\lasso)}_2 \\
    & \leq 
    \norm{\lasso - \E\lasso}_2^2 \norm{C_n}_2  \\
    & = 
    \norm{C_n}_2 \sum_{j=1}^p | \lassoj - \mathbb{E}\lassoj|^2,
  \end{align*}
  which goes to zero uniformly, as $\norm{C_n}_2 \rightarrow \norm{C}_2< \infty$
\end{proof}

\subsection{Concentration of measure for quadratic forms}
\label{sec:conc-quadr-forms}

Finally, we present a
special case of Theorem 1 in \citep{HsuKakade2011} 
which will allow us to prove that part $(c)$ in the decomposition
 converges to zero in probability.

\begin{lemma}
  \label{lem:subg-quad}
  Let $Z\in\R^n$ be a random vector with mean vector $\mu$ satisfying
  \[
  \E\left[ \exp \left(\alpha^\top (Z-\mu)\right)\right] \leq
  \exp\left( \norm{\alpha}^2\tau^2/2\right)
  \]
  for some $\tau>0$ and all $\alpha\in \R^n$. Then, for all $\epsilon>0$,
  \[
  P\left( \left| \frac{1}{n} Z^\top Z - \norm{\mu}^2 - \sigma^2\right| >
    \epsilon \right) \leq 2e^{-n\epsilon^2}.
  \]
\end{lemma}
\begin{proof}
  This result follows from a result in \citep{HsuKakade2011} (see
  also \citep{HansonWright1971}) which we have included in the appendix.
  By that result with $A=I$, we have
  \begin{align}
    P\left(  \frac{1}{n} Z^\top Z - \norm{\mu} - \tau^2
      > 2\sqrt{\frac{t}{n}} \left( \tau^2\left(1+2\sqrt{\frac{t}{n}}\right) +
        \norm{\mu}^2\right) \right) \leq e^{-n\sqrt{\frac{t}{n}}^2}
  \end{align}
  Setting $\delta = \sqrt{t/n}$ and $\epsilon=2\delta \left( \tau^2\left(1+2\delta\right) +
    \norm{\mu}^2\right)$, we can solve for $\delta$. The quadratic
  formula gives (under the constraint $\delta>0$)
  \[
  \delta = \frac{\sqrt{(\tau^2 + \norm{\mu}^2)^2 +
      4\tau^2\epsilon} - \tau^2 -\norm{\mu}^2} {4\tau^2} \geq \epsilon
  \]
  by concavity of $\sqrt{\cdot}$.
  Thus, for any $\epsilon>0$,
  \[
  P\left(  \frac{1}{n} Z^\top Z - \norm{\mu} - \tau^2
    > \epsilon \right) \leq e^{-n\delta^2} \leq e^{-n\epsilon^2}
  \]
  The same argument can be applied symmetrically. A union bound
  gives the result.
\end{proof}

\section{Proofs}
\label{sec:proofs}
In this section, we address each component of the decomposition in
\eqref{eq:decomp}. Parts $(a)$ and $(b)$ follow from
uniform convergence of the lasso estimator to its expectation
(Lemma~\ref{lem:a1a}) and asymptotic equivalence of the
leave-one-out lasso estimator and the full-sample lasso estimator
(Lemma~\ref{lem:lynchpin}) while
part $(c)$ requires the sub-Gaussian concentration of measure result
in Lemma~\ref{lem:subg-quad}.

\begin{proposition}[Part $(a)$] 
  \label{thm:alpha}
  \[
  \sup_{\lambda \in \Lambda} \left|\frac{1}{n} \E \norm{\X\lasso}_2^2 - \frac{1}{n}
    \sum_{i=1}^n \left(X_i^{\top}\thetaLOOnoj\right)^2  \right| =
  \op(1). 
  \]
\end{proposition}

\begin{proof}
  Observe
  \begin{align*}
    \lefteqn{\left| \frac{1}{n}\E ||\X\lasso||_2^2 
        - \frac{1}{n} \sum_{i=1}^n (X_i^{\top}\thetaLOOnoj)^2  \right| }\notag\\
    & \leq
    \underbrace{\bigg|\frac{1}{n} \E ||\X\lasso||_2^2 -
      \frac{1}{n} ||\X\lasso||_2^2 \bigg|}_{(ai)}
    + 
    \underbrace{\bigg|\frac{1}{n}\norm{\X\lasso}_2^2 - \frac{1}{n} \sum_{i=1}^n
      (X_i^{\top}\thetaLOOnoj)^2  \bigg|}_{(aii)}
  \end{align*}
  
  For $(ai)$, note that $\E ||\X\lasso||_2^2 = \textrm{trace}(\X^{\top}\X\V \lasso) + \norm{\X\E \lasso}_2^2$. 
  Hence, 
  \begin{align*}
    (ai) 
    & \leq 
    \left| \textrm{trace}(C_n\V \lasso)\right| +
    \frac{1}{n}\left|\norm{\E \X\lasso}_2^2 - \norm{\X\lasso}_2^2\right|
    \\ 
    & \leq 
    \norm{C_n}_F\norm{\V \lasso}_F +
    \frac{1}{n}\left|\norm{\E \X\lasso}_2^2 - \norm{\X\lasso}_2^2\right|
    \\ 
    & \leq 
    \sigma^2\norm{C_n}_F\norm{(\X^\top\X)^{-1}}_F +
    \frac{1}{n}\left|\norm{\E \X\lasso}_2^2 - \norm{\X\lasso}_2^2\right|
    \\ 
    & = 
    \frac{\sigma^2}{n}\norm{C_n}_F\norm{C_n^{-1}}_F +
    \norm{\lasso + \E\lasso}_{C_n}\norm{\lasso - \E\lasso}_{C_n}. 
  \end{align*} 
  This term  goes to zero uniformly by Lemma~\ref{lem:a1a}.  The third
  inequality follows from \citep[equation 4.1]{OsbornePresnell2000}. 
  For  $(aii)$, note that 
  \begin{align}
    \frac{1}{n}  \left |||\X\lasso||_2^2 -\sum_{i=1}^n
      (X_i^{\top}\thetaLOOnoj)^2  \right|  
    & =
    \frac{1}{n}\left| \sum_{i=1}^n \left( 
        (X_i^{\top} \lasso)^2
        -  (X_i^{\top}\thetaLOOnoj)^2  \right) \right| \notag \\
    & \leq
    \frac{1}{n}\sum_{i=1}^n \left|
      (X_i^{\top} \lasso)^2
      -  (X_i^{\top}\thetaLOOnoj)^2  \right| \notag \\
    & =
    \frac{1}{n}\sum_{i=1}^n \left|
      X_i^{\top} \lasso\lasso^{\top}X_i
      -  X_i^{\top}\thetaLOOnoj\thetaLOOnoj^{\top}X_i  \right| \notag \\
    & =
    \frac{1}{n}\sum_{i=1}^n \left|
      X_i^{\top}\left( \lasso\lasso^{\top} -
        \thetaLOOnoj\thetaLOOnoj^{\top} \right)X_i  \right| \notag \\ 
    &\leq\frac{1}{n}\sum_{i=1}^n \norm{X_i}_2^2
    \norm{\lasso\lasso^{\top} -
      \thetaLOOnoj\thetaLOOnoj^{\top}}_{F} \label{eq:maxNormDiff}.
  \end{align}
  The
  term $||X_i||_2^2\leq C_X^2$ by Assumption B.  Furthermore,
  \begin{align*}
    \lefteqn{\norm{\lasso\lasso^{\top} -
      \thetaLOOnoj\thetaLOOnoj^{\top}}_F}\\
    & = \norm{\lasso}_2^4 + \norm{\thetaLOOnoj}_2^4 -2 (\lasso^\top
    \thetaLOOnoj)^2\\
    &= \lasso^\top \left(\lasso - \thetaLOOnoj\right) \left(\lasso +
      \thetaLOOnoj\right) +\\
    &\quad+ \thetaLOOnoj^\top\left(\thetaLOOnoj-\lasso\right) \left(\thetaLOOnoj+\lasso\right)\\
    &\leq\left(\norm{\lasso}_2+\norm{\thetaLOOnoj}_2\right)\norm{\lasso+
      \thetaLOOnoj}_2 \norm{\lasso-\thetaLOOnoj}_2\\
    &\leq\left(\norm{\ols}_2+\norm{\hat\theta^{(i)}(0)}_2\right)\norm{\lasso+
      \thetaLOOnoj}_2 \norm{\lasso-\thetaLOOnoj}_2
  \end{align*}
Hence,  by Lemma~\ref{lem:lynchpin}, equation \eqref{eq:maxNormDiff}
goes to zero in probability 
uniformly over $\lambda \in \Lambda$.
\end{proof}
\begin{proposition}[Part $(b)$] %
  \label{thm:bravo}
  \[
  \sup_{\lambda\in\Lambda} \left| \frac{1}{n}\E(\X\lasso)^{\top}\X\theta -
    \frac{1}{n} \sum_{i=1}^nY_iX_i^{\top}\thetaLOOnoj \right| =
  \op(1). 
  \]
\end{proposition}

\begin{proof}

Observe,
  \begin{align}
    \sum_{i=1}^n Y_iX_i^{\top}\thetaLOOnoj  
    & = 
    \sum_{i=1}^n (X_i^{\top} \theta + \sigma^2 W_i)(X_i^{\top}\thetaLOOnoj  ) \\
    & = 
    \sum_{i=1}^n X_i^{\top} \theta X_i^{\top}\thetaLOOnoj + \sum_{i=1}^n
    \sigma^2 W_i X_i^{\top}\thetaLOOnoj. 
  \end{align}
  So,
  \begin{align*}
    \lefteqn{\left| \frac{1}{n}\E(\X\lasso)^{\top}\X\theta - \frac{1}{n}
        \sum_{i=1}^nY_iX_i^{\top}\thetaLOOnoj  \right| }\\
    & \leq \left| \E\lasso^{\top}C_n\theta -
      \lasso^{\top}C_n\theta \right|  +
    \left|\lasso^{\top}C_n\theta - \frac{1}{n}
      \sum_{i=1}^n Y_iX_i^{\top}\thetaLOOnoj  \right| \\ 
    & = \left|(\E \lasso -
      \lasso)^{\top}C_n\theta \right| +
    \left|\lasso^{\top}C_n\theta - \frac{1}{n}
      \sum_{i=1}^n Y_iX_i^{\top}\thetaLOOnoj  \right| \\ 
    & \leq \underbrace{\norm{\E \lasso -
        \lasso}_{C_n} \norm{\theta}_{C_n}}_{(bi)} +  
    \underbrace{\left| \frac{1}{n}\lasso^\top\X^{\top}\X\theta  -
        \frac{1}{n}\sum_{i=1}^n X_i^{\top} \theta X_i^{\top}\thetaLOOnoj
      \right|}_{(bii)} + \\ 
    & \qquad +  \underbrace{\left|\frac{1}{n}\sum_{i=1}^n \sigma^2 W_i
        X_i^{\top}\thetaLOOnoj  \right|}_{(biii)}.
  \end{align*}
  By Lemma \ref{lem:a1a}, $(bi)$ goes to zero uniformly.   For
  $(bii)$,
  \begin{align*}
    \frac{1}{n}\left| \lasso^\top\X^{\top}\X\theta  - \sum_{i=1}^n
      X_i^{\top} \theta X_i^{\top}\thetaLOOnoj \right| 
    & =
    \frac{1}{n}  \left| \sum_{i=1}^n \theta^{\top} X_iX_i^{\top}
      \left( \lasso  - \thetaLOOnoj\right) \right|\\
    & \leq
    \frac{1}{n}\sum_{i=1}^n \left( || \theta ||_2 \norm{X_i}_2^2
      \norm{\lasso  - \thetaLOOnoj}_2\right)\\
    & \leq
    C_{\theta} C_X^2\frac{1}{n}\sum_{i=1}^n 
    \norm{\lasso  - \thetaLOOnoj}_2 .
  \end{align*}
  This goes to zero uniformly by Lemma~\ref{lem:lynchpin}.

  For $(biii)$, $||\thetaLOOnoj||_1 \leq ||\hat\theta^{(i)}(0)||_1$ for any $\lambda,i$. So:
  \begin{align*}
    \left|\frac{1}{n}\sum_{i=1}^n \sigma^2 W_i X_i^{\top}\thetaLOOnoj
    \right|  
    & = 
    \frac{\sigma^2}{n}\left|\sum_{i=1}^n  W_i X_i^{\top}\thetaLOOnoj
    \right|  \\ 
    & \leq
    \frac{\sigma^2}{n}\left|\sum_{i=1}^n  W_i \norm{X_i}_\infty
      \norm{\thetaLOOnoj}_1  \right|  \\ 
    & \leq     \frac{\sigma^2C_X}{n}\left|\sum_{i=1}^n  W_i 
      \norm{\hat\theta^{(i)}(0)}_1 \right| \xrightarrow{ae} 0.
    \end{align*}
    The proof of almost-everywhere convergence is given in the
    appendix. This completes the
    proof of Proposition~\ref{thm:bravo}.
\end{proof}

\begin{proposition}[Part $(c)$] %
  \label{thm:charlie}
  \[
  \left| ||\X\theta||_2^2  + \sigma^2 -  \frac{1}{n} \sum_{i=1}^n Y_i^2\right|= \op(1).
  \]
\end{proposition}

\begin{proof}
  By assumption, $E_{P_i}\left[e^{tW_i}\right] \leq e^{\tau^2 t^2/2}$
  for all $t\in\mathbb{R}$. Thus, for any $\alpha\in\R^n$,
  \begin{align}
    \E\left[ \exp\left( \alpha^\top (Y-\X\theta)\right) \right] &=
    \E\left[ \exp\left( \sum_{i=1}^n \alpha_i (Y_i-X_i^\top
        \theta)\right) \right] \\
    &= \E\left[ \exp\left( \sum_{i=1}^n \alpha_i W_i \right) \right]\\ 
    &=\prod_{i=1}^n E_{P_i}\left[ \exp\left(\alpha_i W_i \right) \right] \\
    & \leq \prod_{i=1}^n \exp\left(\alpha_i^2 \tau^2/2 \right)\\
    &= \exp\left(\norm{\alpha}_2^2 \tau^2/2\right).
  \end{align}
  Therefore, we can apply Lemma \ref{lem:subg-quad} with
  $\mu=\X\theta$. 
\end{proof}

By Propositions~\ref{thm:alpha}, \ref{thm:bravo} and~\ref{thm:charlie}, each term
in \eqref{eq:decomp} converges uniformly in probability to zero thus
completing the proof of Theorem~\ref{thm:mainTheorem}.

\section{Discussion and future work}
\label{sec:discussion}

A common practice in data analysis is to estimate
the coefficients of a linear model with the lasso
and choose the regularization parameter by cross-validation.
Unfortunately, no definitive theoretical results 
existed as to the effect of choosing the tuning parameter in this data-dependent
way. In this paper, we provide a solution to the
problem by demonstrating, under particular assumptions
on the design matrix, that the lasso is risk consistent
even when the tuning parameter  is
selected via leave-one-out cross-validation.

However, a number of
important open questions remain. The first is to generalize to other
forms of cross-validation, especially $K$-fold. In fact, this
generalization should be possible using the methods developed
herein. Lemma~\ref{lem:lynchpin} holds when more than
one training example is held out, provided that the size of the
datasets used to form the estimators still increases to infinity with
$n$. Furthermore, with careful accounting of the held out sets,
Proposition~\ref{thm:bravo} should hold as well.

A second question is to determine whether cross-validation holds in
the high-dimensional setting where $p>n$. However, our methods 
are not easily extensible to this setting. We rely heavily on
Assumption A which says that $n^{-1}\X^\top\X$ has a positive definite
limit as well as the related results of~\citep{FuKnight2000} which are
not available in high dimensions or with random design. 
Additionally, an interesting relaxation of our results would be to assume that
the matrices $C_n$ are all non-singular, but tend to a singular limit.  This would 
provide a more realistic scenario where regularization is more definitively useful.

Finally, one of the main benefits of lasso is its ability to induce
sparsity and hence perform variable selection. While 
selecting the correct model is far more relevant in high dimensions,
it may well be desirable in other settings as well. As mentioned in
the introduction, various authors have shown that cross-validation and
model selection are in some sense incompatible. In particular,  CV
is inconsistent for model selection. Secondly, using prediction
accuracy (which is what $\hat{R}_n(\lambda)$ is estimating) as the
method for choosing $\lambda$ fails to recover the sparsity pattern
even under orthogonal design. Thus, while we show that the
predictions of the model are asymptotically equivalent to those with
the optimal tuning parameter, we should not expect to have the
correct model even if $\theta$ were sparse. In particular,
$\hat{\theta}(\lambda)$ does not necessarily converge to the OLS estimator, and
may not converge to $\theta$. We do show (Lemma~\ref{lem:a1a}) that
$\hat{\theta}(\lambda)$ converges to its expectation uniformly for all $\lambda$. While this
expectation may be sparse, it may not be. But we are unable to show
that with cross-validated tuning parameter, the lasso will select the
\emph{correct} model. While this is not surprising in light of
previous research, neither is it
comforting. The question of whether lasso with
cross-validated tuning parameter can recover an unknown sparsity
pattern remains open.  Empirically, our experience is that
cross-validated tuning parameters lead 
to over-parameterized estimated models, but this has yet to be
validated theoretically.

\appendix

\section{Supplementary results}
\label{sec:suppl-results}

\begin{theorem}
  [Theorem 1 in \citep{HsuKakade2011}]
  Let $A\in \R^{m\times n}$, and define $\Sigma= A^\top A$. Suppose
  that $Z\in \R^n$ is a random vector such that there exists
  $\mu\in\R^n$ and $\sigma>0$ with
  \[
  \E\left[ \exp\left(\alpha^\top(Z-\mu)\right) \right] \leq \exp
  \left(\norm{\alpha}_2^2 \sigma^2 /2 \right)
  \]
  for all $\alpha \in \R^n$. Then, for all $t>0$,
  \[
  P\left( \norm{AZ}_2^2 > g_\sigma(t) + g_\mu (t) \right) \leq e^{-t},
  \]
  where
  \begin{align*}
    g_\sigma(t) &=\sigma^2 \left( \tr(\Sigma) + 2\sqrt{
        \tr(\Sigma^2)t } + 2\norm{\Sigma} t\right)\\
    \intertext{and}
    g_\mu (t) &= \norm{A\mu}^2 \left( 1 + 4\left(
        \frac{t \norm{\Sigma}_2^2 }{\tr(\Sigma^2)} \right)^{1/2} +
      \frac{4t \norm{\Sigma}_2^2 }{\tr(\Sigma^2)} \right)^{1/2},
  \end{align*}
  with $\norm{\Sigma}_2$ the operator norm of $\Sigma$.
\end{theorem}

\begin{proof}[Almost everywhere convergence of $(biii)$]
  \begin{align*}
    \frac{\sigma^2C_X}{n}\left|\sum_{i=1}^n  W_i 
      \norm{\hat\theta^{(i)}(0)}_1 \right|  
    & \leq
    \frac{\sigma^2C_X}{n}\left|\sum_{i=1}^n  W_i
      (\norm{\hat\theta^{(i)}(0) - \theta}_1+
      \norm{\theta}_1) \right|  \\  
    & \leq
    \frac{\sigma^2C_X}{n}\left(\left|\sum_{i=1}^nW_i\norm{\hat\theta^{(i)}(0)
          - \theta}_1\right| \label{eq:ugly} 
      + \left|C_\theta \sum_{i=1}^n W_i \right|  \right)
  \end{align*}
  The second term goes to zero in probability by the strong law of large numbers. 
  For the first term, define $\Cin = (n-\norm{X_i}_2^2)^{-1}$,
  then
  \begin{align*}
    \hat \theta^{(i)}(0) 
    & = (\X_{(i)}^{\top}\X_{(i)})^{-1} \X_{(i)}^{\top} Y_{(i)} \\
    & = \ols - \Cin X_i Y_i + \Cin X_iX_i^\top \ols\\
    & = \frac{1}{n}(\Cin X_i X_i^\top +I)\X^\top \X \theta +
    \frac{1}{n} (\Cin X_i X_i^\top +I) \X^\top W\\
    &\quad- \Cin X_i X_i^\top
    \theta - \Cin X_i W_i\\
    & = \theta + \frac{1}{n}(\Cin X_i X_i^\top +I) \X^\top W - \Cin X_iW_i.\\
    \intertext{So,}
    \lefteqn{\norm{\hat\theta^{(i)}(0) - \theta}_1}\notag \\ 
    & = 
    \norm{ \frac{1}{n}(\Cin X_i X_i^\top +I) \X^\top W - \Cin X_iW_i }_1 \\
    & =
    \norm{ \frac{1}{n}\Cin  X_i X_i^\top \X_{(i)}^\top W_{(i)} + \frac{1}{n}\Cin X_i
      X_i^\top X_i W_i +\frac{1}{n} \X^\top W - \Cin X_iW_i }_1 \\
    & =
    \frac{1}{n} \norm{ \Cin X_i X_i^\top \X_{(i)}^\top W_{(i)} - X_iW_i
      +\X^\top W }_1 \\
    &=
    \frac{1}{n}\norm{ \Cin X_i X_i^\top\X_{(i)}^\top W_{(i)} + \X_{(i)}^\top W_{(i)} }_1.
    \intertext{Therefore,}
    \lefteqn{\frac{\sigma^2C_X}{n}\left|\sum_{i=1}^nW_i\norm{\hat\theta^{(i)}(0)
        - \theta}_1\right|}\notag\\
    &= \frac{\sigma^2C_X}{n^2}\left| \sum_{i=1}^n W_i \norm{ \Cin
          X_i X_i^\top\X_{(i)}^\top W_{(i)} + \X_{(i)}^\top W_{(i)}
        }_1\right|\\
    & =   \frac{\sigma^2C_X}{n^2} \left| \sum_{i=1}^n W_i \sum_{k=1}^p
      \left|\Cin \X_{ik}\sum_{\ell=1}^p \X_{i\ell} \sum_{j\neq
          i} \X_{j\ell} W_j + \sum_{j\neq i} \X_{jk} W_j  \right|\right|
    \\
    & =   \frac{\sigma^2C_X}{n^2} \left| \sum_{i=1}^n W_i \sum_{k=1}^p
      \left|\sum_{j\neq i} W_j \left(\Cin \X_{ik} \sum_{\ell=1}^p
          \X_{i\ell}\X_{j\ell} + \X_{jk}\right)  \right|\right| 
    \\
    & \leq   \frac{\sigma^2C^3_X}{n^2} \left| \sum_{i=1}^n W_i \sum_{k=1}^p
      \left|\sum_{j\neq i} W_j\Cin \X_{ik}\right|\right| +
    \frac{\sigma^2C_X}{n^2} \left| \sum_{i=1}^n W_i \sum_{k=1}^p
      \left|\sum_{j\neq i} W_j  \X_{jk}  \right|\right|  \\
    & \leq   \frac{\sigma^2C^4_XC_n^*}{n^2} \left| \sum_{i=1}^n W_i
      \left|\sum_{j\neq i} W_j\right|\right| +
    \frac{\sigma^2C^2_X}{n^2} \left| \sum_{i=1}^n W_i
      \left|\sum_{j\neq i} W_j  \right|\right|,
  \end{align*}
  where $C^*_n := (n-\max_i\norm{X_i}_2^2)^{-1} = \max_i \Cin$.
  To bound $\frac{1}{n^2}\left|\sum_{i=1}^nW_i|\sum_{j\neq i}  W_{j} |\right|$, observe
  \begin{align*}
    \frac{1}{n^2}\left|\sum_{i=1}^nW_i \left|\sum_{j=1}^n  W_{j}  - W_{i}\right|\right|
    & \leq
    \frac{1}{n^2}\left|\sum_{i=1}^nW_i \left( \left|\sum_{j=1}^n
          W_{j}\right|  + |W_{i}|\right)\right| \\ 
    & \leq
    \frac{1}{n^2}\left|\sum_{i=1}^nW_i |W_{i}| \right|+
    \frac{1}{n^2}\left|\sum_{i=1}^nW_i \left|\sum_{j=1}^n
        W_{j}\right|\right| \\ 
    & \leq
    \frac{1}{n^2}\left|\sum_{i=1}^nW_i |W_{i}| \right|+
    \frac{1}{n^2}\sum_{i=1}^n \left|W_i \right| \left|\sum_{j=1}^n
      W_{j}\right| \\ 
    & =
    \frac{1}{n}\left|\frac{1}{n}\sum_{i=1}^nW_i |W_{i}|
    \right|+\frac{1}{n}\sum_{i=1}^n  \left|W_i \right|
    \left|\frac{1}{n}\sum_{j=1}^n  W_{j}\right|
    \stackrel{ae}{\rightarrow} 0. 
  \end{align*}
\end{proof}

\bibliographystyle{mybibsty}
\bibliography{lassorefsECML}

\end{document}